\documentclass[a4paper,12pt]{article}
\usepackage[utf8]{inputenc}
\usepackage[english]{babel}
\usepackage[ruled,section]{algorithm}
\usepackage{makeidx}
\usepackage[font=small]{caption}
\usepackage{amsmath}
\usepackage{amsthm}
\usepackage{amsfonts}
\usepackage{amssymb}
\usepackage{mathrsfs}
\usepackage{graphicx}
\usepackage[margin=25mm]{geometry}
\usepackage{enumerate}
\usepackage{indentfirst}
\usepackage[none]{hyphenat}
\usepackage{color}
\usepackage{authblk}
\makeindex
\newcommand{\real}{\mathbb{R}}
\newcommand{\n}{\mathbb{N}}
\newcommand{\z}{\mathbb{Z}}
\newcommand{\rN}{ {\mathbb{R}^N} }

\newcommand{\rmd}{\; \mathrm{d}}
\numberwithin{equation}{section}
\newtheorem{teo}{Theorem}[section]
\newtheorem{lem}[teo]{Lemma}
\newtheorem{prop}[teo]{Proposition}
\newtheorem{corol}[teo]{Corollary}

\theoremstyle{definition}

\newtheorem{rmk}[teo]{Remark}
\newtheorem{example}[teo]{Example}

\newtheorem{defin}[teo]{Definition}
\newtheorem{conjecture}[teo]{Conjecture}

\newtheorem{claim}[teo]{Claim}
\newcommand{\R}{\mathbb{R}}
\newcommand{\eps}{\varepsilon}
\newcommand{\pa}{\partial}

\usepackage[colorlinks=false,backref=true,pagebackref=false,pdftex,pdfauthor={Ederson, Gabrielle,Nicola},pdftitle={Unique continuation principles for systems}]{hyperref}

\begin{document}

\title{\bf\Large On unique continuation principles for some elliptic systems}
\author[1]{Ederson Moreira dos Santos\footnote{ederson@icmc.usp.br}}
\author[1]{Gabrielle Nornberg\footnote{gabrielle@icmc.usp.br}}
\author[2]{Nicola Soave\footnote{nicola.soave@polimi.it}}
\affil[1]{\small Instituto de Ciências Matemáticas e de Computação, Universidade de São Paulo, Brazil}
\affil[2]{\small Dipartimento di Matematica, Politecnico di Milano, Italy}
\date{}

\maketitle

{\small\noindent{\bf{Abstract.}}
In this paper we prove unique continuation principles for some systems of elliptic partial differential equations satisfying a suitable superlinearity condition.
As an application, we obtain nonexistence of nontrivial (not necessarily positive) radial solutions for the Lane-Emden system posed in a ball, in the critical and supercritical regimes.
Some of our results also apply to general fully nonlinear operators, such as Pucci’s extremal operators, being new even for scalar equations.
}
\medskip

{\small\noindent{\bf{Keywords.}} {Unique continuation; Elliptic system; Lane-Emden; Nonexistence.}

\medskip

{\small\noindent{\bf{MSC2010.}} {35J47, 35J60, 35B06, 35B50.}

\medskip

\noindent \textbf{{Acknowledgments.}} We thank an anonymous referee for the careful reading of the manuscript, and for precious suggestions. \\
Ederson Moreira dos Santos was partially supported by CNPq grant 307358/2015-1; Gabrielle Nornberg was supported by FAPESP grants 2018/04000-9 and 2019/031019-9, São Paulo Research Foundation; Nicola Soave was partially supported by the INDAM-GNAMPA project ``Esistenza e propriet\`a qualitative per soluzioni di EDP non lineari ellittiche e paraboliche".

\section{Introduction and main results}\label{Introduction}

We investigate unique continuation properties for systems of the form
\begin{align} \label{gen syst}
\left\{
\begin{array}{rc}
\Delta u +f(x,u,v) = 0 &\mbox{ in }  \Omega \\
\Delta v +g(x,u,v) = 0 &\mbox{ in }  \Omega ,
\end{array}
\right.
\end{align}
where $\Omega$ is a domain in $\rN$. Here, $f,g: \Omega \times \R^2 \to \R$ are Carath\'eodory functions, that is, continuous in $(u,v)$ for a.e.\ $x$ and measurable in $x$ for every $(u,v)$. Further, they satisfy the following growth condition: for any compact set $K_1 \times K_2  \subset \subset \Omega \times \R^2$ there exists $C>0$ such that
\begin{equation}\label{growth}
|f(x,u,v)| \le C |v|^q, \quad |g(x,u,v)| \le C |u|^p \quad \text{for a.e. $x \in K_1$, for every $(u,v) \in K_2$},
\end{equation}
with $p,q$ satisfying the superlinearity assumption
\begin{align}\label{superlinear}
 p,q>0, \quad \text{and} \quad pq>1.
\end{align}
We stress that either $p$ or $q$ may be strictly smaller than $1$.

The model case we have in mind is the following weighted Lane-Emden system
\begin{align} \label{LE}
\begin{cases}
\Delta u +|x|^\beta |v|^{q-1}v = 0 & \text{ in }\Omega \\
\Delta v +|x|^\alpha |u|^{p-1}u = 0 & \text{ in }\Omega, \end{cases}  \qquad \alpha,\beta \geq 0, \ p,q>0, \ pq>1,
\end{align}
but we can go much further, including locally bounded solutions to more general Emden-Fowler systems, namely with $f(x,u,v) = |x|^{\beta}|u|^{r-1}u |v|^{q-1}v$ and $g(x, u,v) = |x|^{\alpha}|u|^{p-1}u|v|^{s-1}v$, with $p,q,r,s,\alpha, \beta \geq 0$ and $pq>1$.

Unique continuation principles (UCP) are fundamental and of independent interest in the theory of elliptic partial differential equations.
Applications regarding the vanishing of a solution are often associated for instance with solvability, stability, geometrical properties of solutions and so on.

Over the years a lot of contributions and variants about UCP were considered. The most common methods are Carleman type estimates, and doubling inequalities obtained by means of monotonicity formulas via Almgren's frequency. We refer to \cite{Carleman, GL86, GL87, Hormander, JeKe, KoTa, Protter} and references therein for a historical overview. All the previous contributions concerns scalar superlinear equations.
See also \cite{Ru, SW} for recent results concerning sublinear equations.

A very related topic to UCP is the structural analysis of nodal sets of solutions. In particular, tools related to regularity and behavior of free boundary problems turned out to be important instruments for establishing UCP. Techniques originally created for linear and superlinear scalar cases from the pionering works \cite{CFJDE79, CFJDE85, Han, Lin} were developed in \cite{STsub, STsing} to treat sublinear and singular scalar regimes.

Regarding systems, in \cite{AH2018} it was proved that the zero set of least energy solutions of Lane-Emden systems with Neumann boundary conditions has zero measure, in the sublinear regime, namely for the problem \eqref{LE} with $\alpha=\beta =0$ and $pq<1$. The respective superlinear case \eqref{superlinear} was left as completely open, even for special solutions.

In this paper we will be interested in exploiting unique continuation results for a rather general class of systems, by featuring a careful treatment of the nodal sets and of the behavior of solutions near appropriate vanishing points.
Our approach is genuinely designed for systems which do not degenerate into a scalar problem. We discuss some cases in which this does happen in the end of the paper, through a classification result which entails a quite large gamma of problems, even of fully nonlinear nature; see Section \ref{section discussion}.

\smallskip

Solutions of the system \eqref{gen syst} will be understood in the strong $W^{2,m}_{\mathrm{loc}}$ sense for $m> N$, i.e.\ with identities holding at almost every point of $\Omega$.
If in addition we assume $u=v=0$ on $\Gamma\subset\partial\Omega$, then differentiability up to $\Gamma$ will also be implied. 
We stress that the $W^{2,m}_{\mathrm{loc}}$ hypothesis is naturally satisfied by locally bounded solutions under \eqref{growth}-\eqref{superlinear}, by local $W^{2,m}$ regularity estimates.

We now proceed with the statement of our main results.

\begin{defin}\label{def vanishing}
We say that a measurable function $u$ defined in $\Omega$ vanishes with infinite order at the point $x_0\in \Omega$ provided that for some $\eps>0$,
\begin{align*}
\int_{B_r(x_0)} |u|^{\eps} \;\mathrm{d} x  =  o(r^m) \;\textrm{ as $r\rightarrow 0^+$, \,for every $m>0$.}
\end{align*}
\end{defin}

\begin{teo}\label{UC general}
Let $\Omega \subset \R^N$ be a domain, and assume that $f,g$ satisfy \eqref{growth} and \eqref{superlinear}.
Let $m>N$ and $u,v \in W^{2,m}_{\mathrm{loc}}(\Omega)$ be a solution of \eqref{gen syst}. If both $u$ and $v$ vanish with infinite order at some point $x_0\in \Omega$, then $u\equiv v \equiv 0$ in $\Omega$.
\end{teo}

Clearly, the theorem applies to solutions of \eqref{LE} under the assumption \eqref{superlinear}, but it is actually much more general. Indeed, while \eqref{LE} is a prototype of Hamiltonian strongly coupled system, in Theorem \ref{UC general} we do not need such structural assumptions.
Here, \emph{strongly coupled} means that, whenever $(u,v)$ is a solution of \eqref{gen syst} and $u \equiv 0$ (respectively $v \equiv 0$) in an open subset $\omega \subset \Omega$, then also $v \equiv 0$ ($u \equiv 0$) in $\omega$.
Our proof of Theorem \ref{UC general} is based on an iterative scheme inspired by a result of Caffarelli and Friedman \cite{CFJDE79}, and in particular does not rely on Carleman estimates or monotonicity formulas (which seem hardly adaptable to general systems).

\begin{rmk}
A statement like Theorem \ref{UC general} cannot hold if we replace the superlinearity condition \eqref{superlinear} by the sublinear one $pq<1$, even dealing with strongly coupled systems. Indeed, if $w$ is a dead-core solution to $w'' = |w|^{q-1}w$ in an interval, with $0<q<1$, then $(u,v) = (w,\pm w)$ solves in the same interval
\[
u'' \mp |v|^{q-1}v =0 \,, \qquad
v'' \mp |u|^{q-1} u = 0\,,
\]
providing a counterexample to UCP. In this perspective, we observe that the result in \cite{AH2018} rests in a crucial way upon the variational characterization of the solution.

The limit case $pq=1$ remains open in general; we refer to Theorem \ref{UCradial} and to Section \ref{section discussion} for some partial results.
\end{rmk}

Unique continuation as stated in Theorem \ref{UC general}, through infinite vanishing order, is usually known as UCP in the strong form. One may wonder whether it is possible to impose a vanishing condition only on one component, and obtain that $u \equiv v \equiv 0$ in $\Omega$. The answer is negative without further assumptions. For instance, $(u(x),v(x)) = (x_1,0)$ solves
\begin{equation}\label{counter1}
\Delta u + f(x,u,v) = 0 \,  , \qquad
\Delta v + u \varphi(v) = 0 \quad \text{ in $\Omega$},
\end{equation}
whenever $\varphi$ is continuous, $\varphi(0) = 0$, and $|f(x,u,v)| \le C |v|^q$ with $q>0$; this choice is compatible with \eqref{growth}-\eqref{superlinear}. If however the system is strongly coupled, then it is immediate to obtain from Theorem \ref{UC general} a weak UCP supposing that only one component is vanishing in an open set.

\begin{corol}
Let $\Omega \subset \R^N$ be a domain, $m>N$, and assume that $f,g$ satisfy \eqref{growth} and \eqref{superlinear}. Assume moreover that system \eqref{gen syst} is strongly coupled, and
let $u,v \in W^{2,m}_{\mathrm{loc}}(\Omega)$ be a solution of \eqref{gen syst}. If either $u \equiv 0$ or $v \equiv 0$ in an open subset $\omega \subset \Omega$, then $u \equiv v \equiv 0$ in $\Omega$.
\end{corol}

Indeed, the vanishing of $u$ in $\omega$ immediately implies that $v \equiv 0$ in $\omega$. In particular, both $u,v$ vanish with infinite order at a point, and Theorem \ref{UC general} applies.

Another standard consequence of Theorem \ref{UC general} is the following UCP in its boundary version.

\begin{corol}\label{UCderivatives}
Let $\Omega$ be a bounded $C^{1}$ domain, $m>N$, and $\Gamma$ be an open subset of $\partial\Omega$. Assume that $f,g $ satisfy \eqref{growth} up to the boundary under \eqref{superlinear}, and let $u,v \in W^{2,m}(\Omega)$ be a solution of \eqref{gen syst}.
If $u=v=0$ and $\partial_\nu u = \partial_\nu v=0$ on $\Gamma$, then $u\equiv v \equiv 0$ in $\Omega$.
\end{corol}

A natural problem consists of studying a stronger version of the UCP assuming that only one between $u$ and $v$ vanishes with infinite order at a point, in the case of strongly coupled systems. In this generality the problem remains open. However, the special case $q=1$ can be used in order to prove the validity of the strong UCP for semilinear fourth order problems.

\begin{teo}\label{UCP fourth}
Let $\Omega$ be a domain of $\R^N$, $m>N$, and $g: \Omega \times \R^2 \to \R$ such that
$
|g(x,t_1, t_2)| \le C |t_1|^p,
$
with $p>1$ and $C>0$. If  $u \in W^{4,m}_{\mathrm{loc}}(\Omega)$ solves
\begin{equation}\label{4or}
\Delta^2 u = g(x,u,\Delta u) \qquad \text{in\;\, $\Omega$},
\end{equation}
and $u$ vanishes with infinite order at a point $x_0 \in \Omega$, then $u \equiv 0$ in $\Omega$.
\end{teo}
Previous results concerning UCP for certain linear fourth order (or even higher order) equations can be found in \cite{AliBao, ColGra, ColKoc}
for $C^\infty$ solutions at $x_0$ (then the request that $u$ vanishes with infinite order at $x_0$ translates into the fact all its derivatives vanish at $x_0$);
see also the references therein. Their proofs rest on Carleman estimates and harmonic analysis techniques. As a counterpart our result offers, in the strictly superlinear case, a different proof under weaker regularity assumptions.

\medskip

In the sequel a pertinent question on a fully nonlinear analogue of Theorem~\ref{UC general} arises. In \cite{ASuc} the authors obtained UCP for viscosity solutions of $C^{1,1}$ operators. However, for instance Pucci's extremal operators -- which play an important role in stochastic control theory -- are not included in that approach, since they are not differentiable. In fact, it was left as an open problem there how to obtain UCP for such operators in the scalar case. Here we give a partial result for this question, by establishing a general UCP for \emph{radial} viscosity solutions in a ball, in the context of systems.
Our approach permits us to treat general fully nonlinear uniformly elliptic operators $F$ in the form
\begin{align}\label{SC}
\mathcal{M}^-_{\lambda,\Lambda}(X-Y)-\gamma |p-q|\leq F(x,p,X)-F(x,q,Y)\leq \mathcal{M}^+_{\lambda,\Lambda}(X-Y)+\gamma |p-q|,
\end{align}
for all $x \in \Omega$, $p,q\in \rN$, $X,Y \in \mathbb{S}^N$, where $\gamma \geq 0$ and $F(\cdot,0,0) \equiv 0$.
Here, $\mathcal{M}^\pm_{\lambda,\Lambda}$ are the Pucci's extremal operators; see Section \ref{section radial} for their definitions, and also for the definition of viscosity solution.

\begin{teo}\label{UCradial}
Let $\Omega=B_R$ be a ball and $F_1,F_2$ operators as in \eqref{SC}.
For $\alpha, \beta \geq0$, and $p,q>0$ with $pq \ge 1$, consider a radial viscosity solution $u,v \in C^1 (\overline{\Omega})$ of
\begin{align} \label{LE pucci}
\begin{cases}
\,F_1 (x, Du, D^2 u) +|x|^\beta |v|^{q-1}v = 0 &\mbox{in } \; \Omega \\
\,F_2 (x, Dv, D^2 v) +|x|^\alpha |u|^{p-1}u = 0 &\mbox{in } \; \Omega . \\
\end{cases}
\end{align}
If $u(R)=v(R)=0$, and either $u^\prime (R)=0$ or $v^\prime (R)=0$, then $u \equiv v \equiv 0$ in $\Omega$.
\end{teo}

Notice that the limit case $pq=1$ is included in our result, differently to what happen in Theorem \ref{UC general}.
In particular, by taking $F_1=F_2$, $\alpha=\beta$, $p=q$, and considering a solution of type $u=v$, Theorem \ref{UCradial} gives a UCP for scalar fully nonlinear equations with linear or superlinear reaction term. 

We also highlight that no regularity assumptions on $F_1,F_2$ are required in Theorem \ref{UCradial}, but differentiability on the solution pair $(u,v)$.
In what concerns sufficient regularity conditions, continuous (up to the boundary) viscosity solutions of \eqref{LE pucci} in fact belong to $C^{1}(\overline{\Omega})$, if for instance $F_i$, $i=1,2$, are continuous up to the boundary in the variable $x$.
Moreover, it is standard to achieve $W^{2,m}$ regularity, for any $m>N$,  in the presence of convexity or concavity assumptions on $F_i$ with respect to the Hessian entry.

Theorem \ref{UCradial} deserves some further comments. Up to our knowledge, our procedure is different from anything  in the literature, and it is based on the Alexandrov-Bakelman-Pucci (ABP) inequality. This is the key to obtain a radial version of UCP in such generality; see also our Proposition \ref{UCradial'} for a more general statement in radial domains. We stress that Caffarelli-Friedman type estimates used to prove Theorem~\ref{UC general} seem not available for operators in nondivergence form.

\medskip

Concerning applications, we highlight that unique continuation type arguments naturally appear as fundamental tools for establishing nonexistence of nontrivial solutions for the Lane-Emden system.
Here we investigate this question in the radial setting, when the pair $(p,q)$ lies on or over the critical hyperbola, namely
\begin{align}\label{supercritico}
\frac{N+\alpha}{p+1}+\frac{N+\beta}{q+1}\leq N-2,  \quad \alpha,\beta \geq 0, \ p,q>0 \quad (N \ge 3).
\end{align}

The corresponding scalar problem has been studied since the '60s, and it is known that in any starshaped domain the critical or supercritical Lane-Emden equation has no nontrivial solution. This is a consequence of the sole Poho\v{z}aev identity in the supercritical case, while for the critical regime it also relies on the boundary UCP. The case of the Lane-Emdem system has turned out to be difficult and, up to our knowledge, only nonexistence of positive or negative solutions are available; see for example  \cite[Proposition 3.1]{Mitidieri93}.

On the other hand, in the subcritical regime, which is the complementary set of \eqref{supercritico}, several authors proved existence of positive solutions for a general bounded regular domain; see for instance \cite{EDR2013} for a rather complete overview on the subject. Uniqueness of positive solution is known when $pq<1$ for general bounded smooth domains; see \cite[Theorem 4.1]{Montenegro}.
In the case of a ball, uniqueness of a positive solution with $\alpha=\beta=0$ follows from \cite{EG}, which complements the previous results obtained for $p,q\geq 1$ via combination of \cite{Troy} and \cite{Dalmasso}. Notice that uniqueness of positive solution cannot be true for general domains if $pq>1$, since multiplicity for an annulus was obtained in \cite{AnelNonuniq} in the scalar case.

In this work we show that nonexistence of nontrivial radial solutions occurs in the supercritical and critical regimes, \emph{independently of the sign of the solutions}, if the problem is posed in a ball.
Notice that \eqref{supercritico} implies the superlinearity \eqref{superlinear}.
Our nonexistence result is in the sequel.

\begin{teo}\label{Tnonexistence}
Let \eqref{supercritico} hold and $\Omega$ be a ball. Then the problem \eqref{LE} with $u=v=0$ on $\partial\Omega$ does not admit any nontrivial radial classical solution.
\end{teo}

The rest of the paper is organized as follows. We address at first the Lane-Emden systems  \eqref{LE} and \eqref{LE pucci} in the radial setting, in Section \ref{section radial}. We prove Theorem \ref{UCradial}, and then we use it as a crucial tool in the proof of Theorem~\ref{Tnonexistence}. Section \ref{section radial}, beyond being original in itself, also serves as motivation for the general case. This is treated in
Section \ref{section general UC pq>1}, where we prove Theorem \ref{UC general}, Corollary \ref{UCderivatives} and Theorem \ref{UCP fourth}. Section~\ref{section discussion} contains further discussions and applications.

\section{The radial setting}\label{section radial}

In this section we prove a UCP for the weighted Lane-Emden type system \eqref{LE pucci} in the radial regime, and exploit it for the nonexistence of solutions.
We are going to see that the radial case somehow features a rather precise control on the nodal sets of $u$ and $v$.

\smallskip

We first recall that the Pucci's extremal operators $\mathcal{M}^\pm_{\lambda,\Lambda}$ are defined at $X$ as the supremum and infimum of linear operators of the form $\mathrm{tr}(A(x)X)$, taken over all symmetric matrices $A$ such that $\lambda I \leq A(x)\leq \Lambda I$. Precisely, for any symmetric matrix $X$ whose spectrum is $\{e_i\}_{1\leq i \leq N}$,
\begin{align}\label{def Pucci}
\textrm{$\mathcal{M}^+_{\lambda,\Lambda}(X)=\Lambda \sum_{e_i>0} e_i +\lambda \sum_{e_i<0} e_i$, \;\;\; $\mathcal{M}^-_{\lambda,\Lambda}(X)=\lambda \sum_{e_i>0} e_i +\Lambda \sum_{e_i<0} e_i$}.
\end{align}

In what follows we set $B_t$ as the ball $\{ x \in \mathbb{R}^N;\, |x| < t\}$, and
\begin{center}
$\mathcal{L}^\pm_{\lambda,\Lambda} \,[u]=\mathcal{M}_{\lambda,\Lambda}^\pm (D^2 u)\pm \gamma |Du|$, \quad for $\gamma\geq 0$.
\end{center}

Next we briefly recall the definition of viscosity solution. For further details we refer to \cite{CCKS}.
Let $F$ satisfy \eqref{SC} and $h(x)\in L^m_{\textrm{loc}}(\Omega)$.
We say that $u\in C(\Omega)$ is an $L^m$-viscosity solution of ${F}(x,Du,D^2u)\ge h(x)$ (i.e.\ subsolution of $F=h$; respectively solution of ${F}(x,Du,D^2u)\le h(x)$ as supersolution of $F=h$) if whenever $\phi\in  W^{2,m}_{\mathrm{loc}}(\Omega)$, $\varepsilon>0$ and $\mathcal{O}\subset\Omega$ open are such that
\begin{align}\label{eq def visc}
{F}(x,D\phi(x),D^2\phi (x))-h(x)  \leq -\varepsilon \qquad \quad
(\textrm{resp.\ }  {F}(x,D\phi(x),D^2\phi (x))-h(x)  \geq \varepsilon \, )
\end{align}
for a.e. $x\in\mathcal{O}$, then $u-\phi$ cannot have a local maximum $($minimum$)$ in $\mathcal{O}$.

If ${F}$ and $h$ are continuous in $x$, and test functions $\varphi$ are taken in the space $C^2(\Omega)$, then \eqref{eq def visc} becomes an evaluation at a point, and we also say that $u$ is a $C$-viscosity sub or supersolution.
A \textit{solution} is always both sub and supersolution of the equation, and solutions of a Dirichlet problem are supposed to be continuous up to the boundary.
A \textit{strong} solution of ${F}(x,Du,D^2u)\ge h(x)$ belongs to $W^{2,m}_{\mathrm{loc}}(\Omega)$ and satisfies this inequality at almost every point; such notion is equivalent to the $L^m$-viscosity one if $u\in W^{2,m}_{\mathrm{loc}}(\Omega)$ for $m>N$, see \cite{CCKS}.

\smallskip

We will say so that \textit{$(u,v)$ is a viscosity solution  of the system} \eqref{LE pucci} if $u$ is an $L^m$-viscosity solution  of the first equation, and $v$ is an $L^m$-viscosity solution of the second one. If $F_i$, $i=1,2$, are continuous in $x$, then we understand it in the $C$-viscosity sense; and if in addition each $F_i$ is either convex or concave in the Hessian entry it will be implied in the strong sense.

\medskip

In view of the application of Theorem \ref{UCradial} to the nonexistence result, we need the following result, which is stronger than Theorem \ref{UCradial}.

\begin{prop}\label{UCradial'}
Let $\Omega$ be a ball or an annulus, and let $R>0$ such that $\partial B_R$ is contained concentrically in $\overline{\Omega}$.
Consider a radial viscosity solution $(u,v) \in C^1(\Omega \cup \partial B_R) \cap C(\overline{\Omega})$ of \eqref{LE pucci} with $\alpha, \beta \ge 0$, $p,q>0$, and $pq\ge 1$, satisfying the Dirichlet boundary condition $u = v = 0$ on $\partial \Omega$. Suppose that $u(R)=v(R)=0$. Then one of the following alternatives holds.
\begin{enumerate}[(i)]
\item $R$ is an isolated zero for both $u$ and $v$ from one side: in this case $u'(R) v'(R) >0$;
\item $R$ is not an isolated zero for either $u$ or $v$ from one side: in this case $u, v \equiv 0$ in $\Omega$.
\end{enumerate}
\end{prop}

\begin{proof}
Say $\overline{\Omega} =\{x\in \overline{B}_t; \,s\leq |x|\leq t\}$, and w.l.o.g. assume $0\leq s<R\leq t$, since the case $R=s$ can be treated similarly. Consider $u,v\in C^1(\Omega \cup \partial B_R) \cap C(\overline{\Omega})$, and recall that $u(R)=v(R)=0$.

\smallskip

$(i)$ In the first case we assume that $r=R$ is an isolated zero for both $u$ and $v$ from one side.\\
With this, we mean that there exists $\eps >0$ such that either $u$ and $v$ have a strict sign (not necessarily the same) in $B_R\setminus \overline{B}_{R-\eps}$, or $u$ and $v$ have a strict sign (not necessarily the same) in $B_{R+\eps}\setminus \overline{B}_{R}$. Without loss of generality, we focus on the former alternative and consider a maximal annulus $A=B_R\setminus \overline{B}_{R-\delta}$ in which $u$ has defined sign (the case when $u$ does not change sign in $B_R \cap \Omega$ is also admissible, we simply let $A=B_R \cap \Omega$). Say $u>0$ in $A$; otherwise replace the pair $(u,v)$ by $(-u,-v)$. Indeed, notice that this is possible because $(-u,-v)$ solves \eqref{LE pucci} for some fully nonlinear uniformly elliptic operators $\tilde{F}_i$, $i=1,2$, where $\tilde{F}_i(x,p,X)=-F_i(x,-p,-X)$.
Then, $u=0$ on $\partial A$.
We also take a maximal annulus in which $v$ has defined sign. By changing if necessary the roles of $u$ and $v$, we can assume that this annulus for $v$ contains $A$. So, $u,v$ have a well defined sign in $A$, with $u=0$ on $\partial A$. If $v < 0$ in $A$, this would contradict the maximum principle applied to the boundary value problem satisfied by $u$, so $v>0$ in $A$.
Then, we can use the H\"opf Lemma in $A$ for both $u$ and $v$ (see \cite[Theorem 2.7]{B2016}) in order to conclude that $\partial_\nu u (R)<0$ and $\partial_\nu v (R)<0$, where $\nu$ is the exterior unit normal on $|x|=R$. Therefore,
$u'(R) v'(R) = \partial_\nu u (R)\cdot \partial_\nu v (R) >0$ in Case 1.

We also point out that this sign condition does not change if we suppose at the beginning that $u<0$ in $A$. Indeed, in this latter case we firstly have $v < 0$ in $A$ as well, whence we deduce $\partial_\nu u (R)>0$ and $\partial_\nu v (R)>0$.

\medskip

$(ii)$ We assume now that $r=R$ is not an isolated zero for either $u$ or $v$ from one side. \\
Without loss of generality, let $u$ do not have $r=R$ as an isolated zero. In this case, it follows that $u(R)=u^\prime (R)=0$. Indeed, there exists a sequence of radii $r_n\neq R$ converging to $R$ such that $u(r_n)=0$, which we can assume to be increasing. Then, by the mean value theorem, there exists $s_n\in (r_n,r_{n+1})$ such that $u^\prime (s_n)=0$. Since $u$ is differentiable up to $\partial B_R$, $u^\prime (R)=0$.

Now we address the following claim.
\begin{claim}\label{claimABP}
There exists $\varepsilon >0$ such that $u,v\equiv 0$ in $A_{\varepsilon}=B_R\setminus \overline{B}_{R-\epsilon}$.
\end{claim}

\begin{proof}
Set $M_1 = \max\{ \sup_{\Omega} (v^+)^{pq-1},\, \sup_{\Omega} (u^+)^{pq-1}\}$, $M_2 = \max\{ \sup_{\Omega} (v^-)^{pq-1},\, \sup_{\Omega} (u^-)^{pq-1}\}$, and $M = \max\{M_1,M_2\}$, which are well defined due to $u,v\in C(\overline{\Omega})$ and the fact that $pq \ge 1$ (if $pq=1$, we directly pose $M_1=M_2=M=1$). Notice that, if $M_1=0$, then $u,v \le 0$ in $\Omega$, and hence by the strong maximum principle (see \cite[Theorem 2.4]{B2016}) either $u,v<0$ in $\Omega$, or at least one between $u$ and $v$ vanishes identically. In the former case we reach a contradiction with the fact the existence of the sequence $\{r_n\}$; thus $u \equiv 0$ in $\Omega$, whence we deduce that $F_1(x,Du,D^2 u) \equiv 0$ in $\Omega$. In turn, by \eqref{LE pucci}, this implies $v \equiv 0$ as well, which completes the proof. The same argument works if $M_2=0$, and therefore we can focus on the case $M_1, M_2>0$ from now on.

Let $\varepsilon_0 \in (0, R-s)$ small enough such that
\begin{align}\label{choice epsilon0 radial}
C_0^{p+1} M |A_{\varepsilon_0}|^{(p+1)/N } \leq 1/2  \ \ \text{and} \ \ C_0^{q+1} M |A_{\varepsilon_0}|^{(q+1)/N } \leq 1/2,
\end{align}
where $C_0>0$ is a fixed constant only depending on $N$ and $t$ which will be specified ahead. Recalling that $r_n \to R^-$, there exists $n_0 \in \mathbb{N}$ such that $r_n\in (R-\varepsilon_0, R)$ for $n\geq n_0$. Let $\eps = R- r_{n_0}$.

\smallskip

At this point we may assume $v(R-\eps) \le 0$, which implies that $u,v\leq 0$ on $\partial A_{\varepsilon}$.
Observe that, if instead $v(R-\eps) >0$, we could consider the pair $(-u,-v)$, which is still a solution of a system of type \eqref{LE pucci} with $-v(r_{n_0})<0$ and such that $M_1,M_2>0$. In the sequel we show that both $u$ and $v$ are nonpositive in $A_{\varepsilon}$.
Notice that $u,v$ satisfy, in the viscosity sense,
\[
\begin{cases}
\,-\mathcal{L}^+_{\lambda,\Lambda} \,[u]\leq -F_1 (x,Du,D^2 u) = |x|^\beta |v|^{q-1}v \le t^{\beta}(v^+)^q \\
\,-\mathcal{L}^+_{\lambda,\Lambda} \,[v]\leq -F_2 (x,Dv,D^2 v) = |x|^\alpha |u|^{p-1}u \le t^{\alpha}(u^+)^p \end{cases}
\]
in $A_{\varepsilon}$, with $u,v\leq 0$ on $\partial A_{\varepsilon}$.
By applying the ABP estimate for both $u$ and $v$ in the scalar version (see \cite[Proposition 3.3]{CCKS} or \cite[Proposition 2.8]{KSite} for instance), we obtain
\begin{align}\label{abp u e v}
\sup_{A_{\varepsilon}} u\leq C_0\, |A_{\varepsilon}|^{1/N} \sup_{A_{\varepsilon}} (v^+)^q\; , \qquad \sup_{A_{\varepsilon}} v \leq C_0\, |A_{\varepsilon}|^{1/N} \sup_{A_{\varepsilon}} (u^+)^p,
\end{align}
where $C_0$ is the constant from ABP that depends only on $N$, $t$, $\alpha$ and $\beta$ (here we used that the constant in ABP inequality remains bounded if the coefficients and the diameter are bounded). If, by contradiction, there exists a point in $A_{\varepsilon}$ at which $u$ is positive, then the suprema $\sup_{A_{\varepsilon}} u$ and $\sup_{A_{\varepsilon}} u^+$ are equal. Thus, combining the two inequalities in \eqref{abp u e v} we obtain
\begin{align}\label{conta bonita}
\sup_{A_{\varepsilon}} v &\leq C_0^{p+1}  |A_{\varepsilon}|^{(p+1)/N }  \sup_{A_{\varepsilon}} (v^+)^{pq-1}\, \sup_{A_{\varepsilon}} v^+ \le C_0^{p+1}  |A_{\varepsilon}|^{(p+1)/N } M_1  \sup_{A_{\varepsilon}} v^+.
\end{align}
By the choice of $\varepsilon < \eps_0$ in \eqref{choice epsilon0 radial}, we infer that $v \leq 0$ in $A_{\varepsilon}$. Then, by the maximum principle applied to the boundary value problem satisfied by $u$, we get $u \leq 0$ in $A_{\varepsilon}$, which contradicts the assumption on the existence of a point in $A_{\varepsilon}$ at which $u$ is positive. A similar contradiction can be obtained supposing the existence of a point in $A_{\varepsilon}$ at which $v$ is positive.

Therefore $u,v \leq 0$ in $A_{\varepsilon}$. Now, since $-\mathcal{L}_{\lambda,\Lambda}^+[u]  \leq 0$ and $u\leq 0$ in $A_{\varepsilon}$, the strong maximum principle yields $u<0$ in $A_{\varepsilon}$ or $u\equiv 0$ in $A_{\varepsilon}$; but the first one cannot occur, since $\{r_n\}$ is an increasing sequence of zeros of $u$ converging to $R$. So $u\equiv 0$ in $A_{\varepsilon}$, which implies $F_1(x,Du,D^2 u) \equiv 0$ in $\Omega$, and in turn, by \eqref{LE pucci}, $v\equiv 0$ in $A_{\varepsilon}$. This finishes the proof of the claim.
\end{proof}

To sum up, so far we showed that if $R$ is not an isolated zero for $u$ or for $v$ from one side, then it is well defined and strictly positive the following quantity,
\begin{align}\label{def maximal}
{\varepsilon}^*=\sup \{\, \varepsilon \in (s,R) \textrm{ such that }  u,v\equiv 0 \textrm{ in } A_{\varepsilon}  \,\}.
\end{align}

Next we show that $\varepsilon^*=R-s$. For this, let us derive a contradiction from $\varepsilon^*<R-s$. Consider $R^*:=R-\varepsilon^*$, and observe that $u(R^*)=v(R^*)=0$, and by $C^1$ regularity $u^\prime (R^*)=v^\prime (R^*)=0$.
Now, case $(i)$ implies that $R^*$ is not an isolated zero from inside of $B_{R^*}(0)$ for either $u$ or $v$. As in Claim \ref{claimABP}, this fact implies that $u,v\equiv 0$ in $B_{R^*}(0)\setminus \overline{B}_{R^*-\epsilon_1}(0)$, for some $\varepsilon_1>0$. But this means that $u,v\equiv 0$ in $A_{\varepsilon^*+\varepsilon_1}$, which in turn contradicts the maximality of $\varepsilon^*$.

\smallskip

To conclude the proof, we have proved that, if $R>0$ is not an isolated zero for $u$ from inside, we have $u \equiv v \equiv 0$ in $\Omega \cap \overline{B}_R$, and $u'(R) = v'(R) = 0$. By Case 1, this implies that $R$ is not an isolated zero for either $u$ or $v$ also from outside. Adapting the previous argument in the annulus $B_t\setminus B_R$, this finally yields $u \equiv v \equiv 0$ in $\Omega$.
\end{proof}

As already observed, the previous lemma directly implies Theorem \ref{UCradial}. Moreover, an application of the homogeneous Dirichlet problem gives rise to the following corollary.

\begin{corol}\label{corol DuDv>0}
Let $\Omega$ be a ball or an annulus. If $(u,v)\in C^1(\overline{\Omega})$ is a nontrivial viscosity radial solution of \eqref{LE pucci} with $\alpha, \beta \ge 0$, $p,q>0$, $pq \ge1$, and $u=v=0$ on $\partial\Omega$, then $\partial_\nu u \cdot \partial_\nu v >0$ on $\partial\Omega$.
\end{corol}

In turn, this provides nonexistence of nontrivial radial solutions in the ball in the supercritical or critical regimes, as stated in Theorem \ref{Tnonexistence}.
Notice that viscosity solutions of the problem \eqref{LE} are in fact classical and differentiable up to the boundary.

\begin{proof}[Proof of Theorem \ref{Tnonexistence}]
Let $(u,v)$ be a nontrivial classical solution of \eqref{LE} with $u=v=0$ on $\partial\Omega$, and let $\Omega = B_R(0)$. Then, by applying in the radial setting the Poho\v{z}aev identity for systems -- see equation (3.4) in \cite{Mitidieri93}, or equation (2.13) in \cite{EDR2013}, we obtain
\begin{align}\label{eqSup}
R\int_{\partial\Omega} \pa_\nu u \cdot \pa_\nu v  \rmd S=\left\{ \frac{N+\alpha}{p+1}+\frac{N+\beta}{q+1} -(N-2)\right\} \int_{\Omega} |x|^{\alpha}|u|^{p+1}  \rmd x \leq 0\,,
\end{align}
where the inequality follows from \eqref{supercritico}. This is in contradiction with Corollary \ref{corol DuDv>0}.
\end{proof}

\section{UCP in the general case and its consequences}\label{section general UC pq>1}

In this section we prove some variations of unique continuation principles for systems in their general forms. To this end, we will use a fundamental lemma due to Caffarelli and Friedman \cite{CFJDE79, CFJDE85} which is stated in what follows.

We use the notation $[\beta]$ to mean the floor of $\beta$, i.e. $[\beta] =\max \{z\in \z ;\, z\leq \beta \}$. Moreover, set $\langle\beta\rangle=\min \{\beta-[\beta], 1+[\beta]-\beta \}$, which is the distance of $\beta$ to the set of integer numbers $\mathbb{Z}$. 

\begin{lem}[\cite{CFJDE79, CFJDE85}]\label{lemma CF}
Let $\beta_0>0$ and $\beta$ be a positive noninteger with $\beta \geq \beta_0$. Let $m>N$ and $v\in  W^{2,m}(B_1)$ be a function satisfying
\begin{align*}
|\Delta v| \leq C_\beta\, |x|^\beta\;\; \textrm{ in $B_1$, \quad with \;} C_\beta\geq 2^\beta.
\end{align*}
Then
$v(x)= P(x)+\Gamma (x)$ in $ B_1$,
where $P(x)$ is a harmonic polynomial of degree $[\beta]+2$, and
\begin{align*}
|\Gamma(x)| \leq CC_\beta \frac{\beta^{N-2}}{\langle\beta\rangle} |x|^{\beta +2}, \quad
|D\Gamma (x)|  \leq CC_\beta \frac{\beta^N}{\langle\beta\rangle} |x|^{\beta +1} \;\;\textrm{ in } B_1,
\end{align*}
for some constant $C$ depending only on $\beta_0$, $N$, $\|v\|_{L^\infty (\partial B_1)}$, and $\|Dv\|_{L^\infty (\partial B_1)}$.
\end{lem}
In Theorem \ref{UC general}, the fact that $u,v\in W^{2,m}_{\mathrm{loc}}(\Omega)$, $m >N$ is used to apply Caffarelli-Friedman Lemma \ref{lemma CF}; see the proofs of \cite[Lemma 3.1, for $N=3$]{CFJDE79} and \cite[Lemma 1.1]{CFJDE85}, where $W^{2,m}$ estimates are employed to get $C^1(\overline{B}_1)$ regularity and enable a Green's representation formula.

We shall set up an iterative scheme based on Lemma \ref{lemma CF}. In order to control the iterations, we use the following statement.

\begin{lem}\label{lemma iterations}
Let $0<q<1<p$ with $pq>1$, and let $\widehat{c}\in (0,\min\{\langle q \rangle,1/10\}]$. Then there exist sequences $\vartheta_k , \gamma_k\in [0,1)$ such that, by defining
\begin{align}
\alpha_1=q, \quad \beta_k =(\alpha_k +2)p - \gamma_k, \quad\alpha_{k+1} = (\beta_{k}+2)q-\vartheta_k, \quad \textrm{ for all }\, k\in\n,
\end{align}
we have
\begin{enumerate}[(i)]
\item $\alpha_k , \;\beta_k$ are increasing in $k$;

\item $\langle\alpha_k \rangle, \langle \beta_k \rangle \geq \widehat{c} \; $ for all $\;k\in \n\;\;($in particular $\alpha_k , \beta_k \not\in \n)$;

\item $A(pq)^k\leq \alpha_k , \, \beta_k \leq B(pq)^k$, for some $A,B>0$ and all $k\in \n$. In particular, $\alpha_k,\beta_k\rightarrow \infty$.
\end{enumerate}
\end{lem}

\begin{proof}
Firstly we show that $(i)$ holds for any choice of $\vartheta_k, \gamma_k\in [0,1)$.
Indeed, for $k\in \n$,
\begin{align*}
\alpha_{k+1}=((\alpha_k +2)p-\gamma_k +2)q-\vartheta_k= (\alpha_k +1)pq +pq-\vartheta_k +(2-\gamma_k)q > \alpha_k,
\end{align*}
since $pq>1>\vartheta_k$ and $\gamma_k<1$, and analogously for $\beta_k$.

Regarding $(ii)$, for a fixed $k\geq 1$, let us consider $x:=(\beta_k +2)q$. If $\langle x \rangle \geq \widehat{c}$, we choose $\vartheta_k=0$.
Assume then that $\langle x \rangle <\widehat{c}$. We have two cases.
If $x-[x]< \widehat{c}$, we choose $\vartheta_k=2\widehat{c}$. Surely, $[x-2\widehat{c}\,]=[x]-1$, and
$x-2\widehat{c}-[ x-2\widehat{c}\,]=x-[x]+1-2\widehat{c} \geq 1-2\widehat{c} > \widehat{c}$, and
$1+[ x-2\widehat{c}\,]-x+2\widehat{c}=[x]-x+2\widehat{c}>\widehat{c}$.
If instead $x-[x]\geq \widehat{c}$, then $[x]+1-x < \widehat{c}$, and we can choose $\vartheta_k = \widehat{c}$ and argue as before.
A similar argument is used to construct $\gamma_k$.

For $(iii)$, we start noting that $\alpha_1=q = A pq$, for $A = 1/p$. Further, if $\alpha_k \geq A(pq)^k$, then
$\alpha_{k+1}=\alpha_k pq +(2-\gamma_k)q+2pq-\vartheta_k \geq A (pq)^{k+1}$.
Analogously, $\beta_1=pq+2p-\gamma_1>pq > A pq$ since $p>1$, and so on.
On the other hand,
$
\alpha_{k+1}\leq (\beta_k+2)q \leq (\alpha_k +2)pq +2q \leq (\alpha_k+4)pq.
$
Thus, using that $\mathrm{log}\,(\alpha_k+4) \leq \mathrm{log}\, \alpha_k +\frac{4}{\alpha_k}$, we obtain
\begin{align*}
\mathrm{log}\, \alpha_{k+1} \leq \mathrm{log} (pq) +\mathrm{log}\, \alpha_k +\frac{4}{A(pq)^k}.
\end{align*}
Iterating these estimates, and using that $pq>1$, we infer that
\begin{align*}
\mathrm{log}\, \alpha_{k+1} \leq k\, \mathrm{log} (pq) +\mathrm{log}\, \alpha_1 +\sum_{i=1}^{k}\frac{4}{A(pq)^i}
\leq (k+1)\,\mathrm{log}(pq) +C,
\end{align*}
and then $\alpha_k \leq e^C (pq)^k$ for all $k\in\n$.
Also, $\beta_k \leq (\alpha_k+2)q \leq (e^C+2)q (pq)^k \leq B (pq)^k$ for all $k\in\n$ by using again that $pq>1$.
\end{proof}

\begin{proof}[Proof of Theorem \ref{UC general}.]
Recall that we consider a solution pair $(u,v)$ to \eqref{gen syst} with $u,v\in W^{2,m}_{\mathrm{loc}}$, under the assumptions \eqref{growth} and \eqref{superlinear}.
For the sake of simplicity, we understood all inequalities ahead in the almost everywhere sense without explicit writing ``a.e." each time.
Up to translation, scaling, and using  a Sobolev embedding, we can suppose $\Omega \supset \supset B_1$, and that $u,v \in W^{2,m}(B_1)\cap C^1(\overline{B}_1)$ vanish with infinite order at $0$. By \eqref{growth}, we have
\begin{equation}\label{ineq deltas}
|\Delta u| \le   C_0 |v|^q, \quad |\Delta v| \le C_0 |u|^p,
\end{equation}
for a positive constant $C_0$, which we can assume to be greater than $1$.
Now, from $u,v\in C^1(\overline{B}_1)$ we have
$|v(x)|\leq \bar{C}|x|$ in $B_1$.
Here $\bar{C}$ denotes the maximum between the Lipschitz constants of $u$ and $v$, and the constants from Lemma \ref{lemma CF} for $u$ and $v$.

Notice that, since $pq>1$, then at least one between $p,q$ is greater than one. By changing the roles of $u$ and $v$ we may suppose $p>1$. On the other hand, about $q$ we first assume $q<1$ and consider the notation on $\alpha_k,\beta_k$ from Lemma \ref{lemma iterations}. By \eqref{ineq deltas} we derive
\begin{align}\label{step 1}
|\Delta u| \le C_0 |v|^q \leq C_{\alpha_1} |x|^{\alpha_1} \quad \textrm{in $B_1$},
\end{align}
where $C_{\alpha_1}= \max\{2^{\alpha_1}, C_0\bar{C}^q\}$. Thus, by Lemma \ref{lemma iterations}, we obtain $u(x)=P_1(x)+\Gamma_1(x)$, where $P_1$ is a harmonic polynomial of degree $[q]+2=2$, and
\[
|\Gamma_1(x)|\leq \frac{\bar{C}}{\langle \alpha_1 \rangle} C_{\alpha_1} \alpha_1^{N-2}|x|^{\alpha_1+2} \le \widetilde{C} C_{\alpha_1} \alpha_1^{N-2}|x|^{\alpha_1+2},
\]
with $\widetilde{C}\,\widehat{c}\geq \bar{C} $ for $\widehat{c}$ as in Lemma \ref{lemma iterations} $(ii)$.
Since $u$ vanishes with infinite order at $0$, we must have $P_1\equiv 0$. Hence $u=\Gamma_1$ and so, using again \eqref{ineq deltas},
\begin{align*}
|\Delta v| \le C_0 |u|^p\leq C_{\beta_1} |x|^{\beta_1},
\end{align*}
where $C_{\beta_1}\geq \max \{2^{\beta_1}, C_0 \widetilde{C}^p \alpha_1^{p(N-2)}C_{\alpha_1}^p\}$. Now, by the combination of Lemma \ref{lemma CF} and the vanishing of infinite order of $v$ at $0$, we obtain
\begin{align*}
|v(x)|\leq \bar{C} C_{\beta_1} \frac{\beta_1^{N-2}}{\langle \beta_1 \rangle} |x|^{\beta_1+2}\leq \widetilde{C} C_{\beta_1} \beta_1^{N-2} |x|^{\beta_1+2}.
\end{align*}
Iterating the previous argument, we deduce the existence of constants $C_{\alpha_k}, C_{\beta_k}>0$ such that
\begin{align*}
&|\Delta u(x)|\leq C_{\alpha_k} |x|^{\alpha_k}, \qquad |u(x)|\leq \widetilde{C}C_{\alpha_k} \alpha_k^{N-2} |x|^{\alpha_k+2},\\
&|\Delta v(x)|\leq C_{\beta_k} |x|^{\beta_k}, \qquad\, |v(x)|\leq \widetilde{C}C_{\beta_k} \beta_k^{N-2} |x|^{\beta_k+2},
\end{align*}
where $C_{\beta_{k}} \geq \max \{ 2^{\beta_{k}}, C_0 (\widetilde{C} C_{\alpha_k} \alpha_k^{N-2})^p\}$, and $C_{\alpha_{k+1}} \geq \max \{ 2^{\alpha_{k+1}}, C_0 (\widetilde{C} C_{\beta_k} \beta_k^{N-2})^q \}$.

\begin{claim}\label{claim constant iteration}
There exists $K_0\geq 1$, $K\geq 2$, and $\sigma\gg 1$ such that we can pick out
$C_{\alpha_k}=C_0K_0 K^{(pq)^k D_k}$, with $D_k=\sigma\sum_{i=1}^{k} \frac{i}{(pq)^i}$, and $C_{\beta_k}=C_0(\widetilde{C}C_{\alpha_k} \alpha_k^{N-2})^p$.
\end{claim}

In order to prove the claim, we have to verify that
\begin{enumerate}[(a)]
\item $C_0K_0 K^{(pq)^k D_k} \geq 2^{\alpha_k}$;

\item $C_0(\widetilde{C} K_0 K^{(pq)^k D_k} \alpha_k^{N-2} )^p \geq 2^{\beta_k}$;

\item $K_0 K^{(pq)^{k+1} D_{k+1}}\geq
(\widetilde{C} (C_0\widetilde{C}K_0 K^{(pq)^k D_k}\alpha_k^{N-2})^p \beta_k^{N-2} )^q$,
\end{enumerate}
for all $k\in \n$, provided $K_0, K, \sigma$ are sufficiently large.

\medskip

Observe that $D_k\geq \frac{\sigma}{pq}$. Then, for every $K \ge 1$, we have that
$K_0 K^{(pq)^k D_k}\geq K_0 K^{(pq)^k \frac{\sigma}{pq}} $.
On the other hand, by Lemma \ref{lemma iterations} $(iii)$, we have $2^{\alpha_k}\leq 2^{B(pq)^k}$. Thus, $(a)$ holds if
$C_0 K_0 K^{(pq)^k \frac{\sigma}{pq}}> 2^{B(pq)^k}$, which is possible for instance if $K_0>1$, $K>2$, and $\sigma> Bpq$.

For $(b)$, notice that $pq>1$ and Lemma \ref{lemma iterations} $(iii)$ imply that
\begin{center}
$\widetilde{C}^p K_0^p K^{p(pq)^k D_k} \alpha_k^{p(N-2)}
\geq  \widetilde{C}^p K_0^p K^{p(pq)^k \frac{\sigma}{pq}}  A^{p(N-2)} (pq)^{kp(N-2)} \ge \widetilde{C}^p K_0^p K^{p(pq)^k \frac{\sigma}{pq}}  A^{p(N-2)}$,
\end{center}
and $2^{\beta_k}\leq 2^{B(pq)^k}$.
Then $(b)$ holds if
$C_0\widetilde{C}^p K_0^p K^{(pq)^k \frac{\sigma}{q} } A^{p(N-2)}\geq 2^{B(pq)^k}$; and we can simply choose $\sigma>Bq$, $K>2$, and $K_0$ large enough depending on $A,p,q$.

Finally, the r.h.s.\ of $(c)$ is less than or equal to
\begin{center}
$C_0^{pq}\widetilde{C}^{(p+1)q} K_0^{pq} K^{(pq)^{k+1} D_k}\alpha_k^{pq(N-2)} \beta_k^{q(N-2)}\leq C_0^{pq}
\widetilde{C}^{(p+1)q} K_0^{pq} K^{(pq)^{k+1} D_k}\left(B (pq)^k\right)^{(p+1)q(N-2)} $.
\end{center}
The latter is less than or equal to $K_0 K^{(pq)^{k+1} D_{k+1}}$ if
$$C_0^{pq} \widetilde{C}^{(p+1)q} K_0^{pq-1} B^{(p+1)q(N-2)} (pq)^{k(p+1)q(N-2)} \leq K^{(pq)^{k+1} (D_{k+1}-D_k)}=K^{\sigma (k+1)},$$
which, by taking the $(k+1)$ roots, is implied by
$C_0^{pq} \widetilde{C}^{(p+1)q} K_0^{pq-1} B^{(p+1)q(N-2)} (pq)^{(p+1)q(N-2)}  \leq K^{\sigma }$.
Taking any $K_0>1$ and $K >2$, it is possible to choose a sufficiently large $\sigma$ in such a way that this holds. This completes the proof of the claim.

Now, using Claim \ref{claim constant iteration}, and the fact that $|x|\leq 1$, we obtain that, for any $k$,
\begin{align*}
|u(x)| \leq C_0\widetilde{C} B^{N-2} (pq)^{k(N-2)}K_0 K^{(pq)^k D }\,|x|^{A(pq)^k},
\end{align*}
where $D=\lim_k D_k<+\infty$, since $pq>1$. Hence, if we take $|x|<K^{- \frac{pq D}{A}} \Leftrightarrow K^D< |x|^{- \frac{A}{pq}}$, then
\begin{align*}
|u(x)| \leq C (pq)^{k(N-2)}|x|^{A(pq)^k(1-{1}/{pq})}=\mathrm{exp} \left\{\log C + k \,\mathrm{log}\, ((pq)^{N-2})+A \left(1-\frac1{pq}\right) (pq)^k \mathrm{log}\,|x| \right\}.
\end{align*}
Thus, by passing to the limit as $k\rightarrow \infty$, we deduce that $|u(x)|\leq C e^{-\infty}=0$, i.e.\
\begin{align}\label{u=0 in Br}
\textrm{$u\equiv 0$\;\; in the ball $B_{\bar r} (0)$, with $\bar r=K^{- \frac{pq D}{A}}$}.
\end{align}
Up to now, we supposed that $q<1$. The case $p>1$ and $q \ge 1$ is actually simpler, and in any case can be reduced to $q<1$ in the following way: if $q\geq 1$, we write $q=q_1+q_2$, where $q_1\in (0,1)$ is such that $pq_1>1$.
Thus, using $|x|\leq 1$, equation \eqref{step 1} becomes
\begin{align*}
|\Delta u| \le C_0 |v|^q \leq C_0 \bar{C}^q |x|^{q_1} \quad \textrm{in } B_1.
\end{align*}
Then we apply Lemma \ref{lemma iterations} with $q$ replaced by $q_1$.
For the second step, since $|x|\leq 1$,
\begin{align*}
|\Delta u| \le C_0 |v|^q \leq C_0 (\widetilde{C}C_{\beta_1} \beta_1^{N-2})^q |x|^{(\beta_1+2)q_1} \quad \textrm{in } B_1,
\end{align*}
and so on. For the iteration procedure we only need a bit care with the constants, which remain the same for the original $q$. But this is just a question of using $pq\geq pq_1$. The rest of the proof carries on the same way in order to obtain \eqref{u=0 in Br} for general $q>0$ satisfying $pq>1$.

Notice that, by \eqref{u=0 in Br} and \eqref{growth}, we have
\begin{equation}\label{from u to v}
-\Delta v = g(x,0,v) =0 \quad \text{in $B_{\bar r}(0)$, and $v$ vanishes with infinite order at $0$}.
\end{equation}
Thus, regularity and UCP for harmonic functions imply that $v \equiv 0$ in $B_{\bar r}(0)$ as well.

\medskip

To complete the proof of UCP we perform in the sequel a usual connectedness argument.
Set
\begin{center}
$\mathcal{N}=\{x\in \Omega; \, u \equiv v \equiv 0 $ in a neighborhood of $x\in \Omega\}.$
\end{center}
Then $\mathcal{N}\neq \emptyset$ by the previous discussion.
Observe also that $\mathcal{N}$ is open by definition, so we need to show that it is also closed in $\Omega$ in order to conclude $\mathcal{N}=\Omega$.

If $\mathcal{N}$ was not closed, there would exist $x_0\in (\partial\mathcal{N}\cap \Omega)\setminus \mathcal{N}$. Consider a small ball centered at $x_0$, namely $B_{2r}(x_0)\subset \Omega$, $r\leq 1$, and take some $\bar{x}\in B_{r/2}(x_0)$ with $\bar{x}\in \mathcal{N}$.
Then $u$ and $v$ vanish in some ball $B_s(\bar{x})$, with $s< r$. Now we claim that $u=v=0$ in $B_{r}(\bar{x})$.
Indeed, define
$$ R = \sup \{ s\in (0,r); \; u\equiv v \equiv 0 \textrm{ in }B_s (\bar{x})\}>0.$$
If $R<r$, we can repeat the argument deriving \eqref{u=0 in Br} to show that $u \equiv  0$, and hence also $v \equiv 0$ (as in \eqref{from u to v}), in any ball of center in $B_{R}(\bar{x})$ with radius $r^\prime$, for some $r^\prime>0$.
Note that achievement of the whole (arbitrary) ball is ensured because the radius only depends on the norm of $u,v$ and their derivatives on the boundary of the ball $B_{2r}(x_0)$. But this contradicts the definition of $R$ as a supremum. So $R =r$, i.e.\ we obtain $u\equiv v\equiv 0$ in $B_r(\bar{x})$, and the claim is proved. In turn, this contradicts the fact that $x_0 \not \in \mathcal{N}$, and hence we finally deduce that $\mathcal{N}=\Omega$, as desired.
\end{proof}

From Theorem \ref{UC general}, it is not difficult to derive the boundary at Corollary \ref{UCderivatives}.

\begin{proof}[Proof of Corollary \ref{UCderivatives}]
Fix some $x_0\in \Gamma$, write $x=(x^\prime, x_N)$, and consider $\Omega\cap B_r (x_0)=\{x\in B_r(x_0); \; x_N > \phi (x^\prime)\}$, for some small $r>0$ and for a $C^{1}$ function $\phi : \mathbb{R}^{N-1}\rightarrow \real$. We extend the domain near $x_0$ by choosing $\psi \in C_c^2 (\real^{N-1})$ with
\begin{center}
$\psi =0 $ \;in $|x^\prime - x_0^\prime|\geq r/2$; \quad $\psi =1$ \;in $|x^\prime - x_0^\prime|\leq r/4$.
\end{center}
That is, we set
$\mathcal{D}=\{x\in B_r(x_0), \; x_N> \phi (x^\prime)-\varepsilon \psi (x^\prime) \} \subset B_r(x_0)$.
Then $\mathcal{D}$ is an open bounded connected set, with $\partial\mathcal{D}\in C^{1}$.
Next, we define the functions $U,V$ which are extensions by $0$ of the original $u,v$ in $\mathcal{D}\setminus \overline{\Omega}$.
Since $u,v\in W^{2,m}(\Omega)\cap  C^1(\overline{\Omega})$ are so that $u=v=0$ and $\partial_{\nu}u = \partial_{\nu}v =0$ on $\Gamma$, then $U,V\in W^{2,m}(\mathcal{D})$ for $m>N$, and $U,V$ satisfy
\begin{center}
$-\Delta U= F(x,U,V)$, \quad $-\Delta V = G(x,U,V)$ in $\mathcal{D}$, \quad with $U=V=0$ on $\partial\mathcal{D}$,
\end{center}
where $F$ and $G$ are extensions by $0$ of $f, g$ in $\mathcal{D}\setminus \overline{\Omega}$ respectively. Since $U=V=0$ in some open ball $B\subset \mathcal{D}\setminus \overline{\Omega}$, then Theorem \ref{UC general} yields $U=V=0$ in $\mathcal{D}$.
\end{proof}

We focus now on the proof of the strong UCP for the semilinear fourth-order equation \eqref{4or}.

\begin{proof}[Proof of Theorem \ref{UCP fourth}]
We apply interior $L^p$ estimates for the biharmonic problem, see \cite[Chapter V]{AgDoNi} or \cite[Theorem 2.20]{GaGrSw}. That is, there exists $C>0$ such that, for every $w \in W^{4,m}(B_{3r}(x_0))$, $m>1$, we have
\[
\sum_{i=0}^4 r^i \|D^i w\|_{L^m(B_r(x_0))} \le C \left( r^4 \|\Delta^2 w\|_{L^m(B_{2r}(x_0))} + \|w\|_{L^m(B_{2r}(x_0))}\right).
\]
Using the growth assumption on $g$, we deduce in particular that, for sufficiently small $r>0$,
\[
\|\Delta u\|_{L^m(B_r(x_0))}  \le C \left( r^2 \| |u|^p \|_{L^m(B_{2r}(x_0))} + \frac{1}{r^2}\|u\|_{L^m(B_{2r}(x_0))}\right).
\]
Hence, if $u$ vanishes with infinite order at $x_0$, it also does $\Delta u$. At this point, letting $v = -\Delta u$, equation \eqref{4or} can be written as a second order system
\[
\begin{cases}
-\Delta u = v & \text{in $\Omega$} \\
-\Delta v = g(x,u,v) & \text{in $\Omega$},
\end{cases}
\]
which satisfies \eqref{growth} with $q=1$ and $p>1$. By the above discussion, both $u$ and $v$ vanishes with infinite order at $x_0$, and thus Theorem \ref{UC general} gives $u \equiv 0$ in $\Omega$.
\end{proof}

\begin{rmk}
The strong vanishing with infinite order at a point  as in the Definition \ref{def vanishing} for both $u$ and $v$ is equivalent to a stronger polynomial vanishing of $u,v$.
Precisely, we say that $w$ has \textit{polynomial vanishing of infinite order at the point $x_0$} if
\begin{align*}
\lim_{|x-x_0|\rightarrow 0} \frac{w(x-x_0)}{|x-x_0|^\beta}=0, \;\; \textrm{ for all }\; \beta\in \n.
\end{align*}
The equivalence of this and Definition \ref{def vanishing} is due to the Local Maximum Principle applied to each equation of the system.
To see this, assume that both $u$ and $v$ vanish with infinite order, say at $x_0=0$, in the sense of Definition \ref{def vanishing}, for some $\varepsilon_u, \varepsilon_v$ there.
Firstly note that $g(x,u(x),v(x))\in L^s(\overline{B}_r)$ for $s>N$ and small $r>0$; assume for instance $ps>\varepsilon_u$. Since $u,v$ are continuous at $0$, we can suppose $|u|, |v|\leq 1$ in some small ball centered at $0$; in particular $|u|^{ps}\leq |u|^{\varepsilon_u}$ there, and so
\begin{align*}
|v(x)|\leq C \left\{ \left(\int_{B_r} |v|^\varepsilon\right)^{1/\varepsilon} + \left(\int_{B_r} |u|^{ps}\right)^{1/s} \right\}=O(r^\beta),
\end{align*}
for all $x\in B_{r/2}$ and $\beta>0$, by taking $\varepsilon=\varepsilon_v$. Analogously we derive a pointwise estimate for $u$.
\end{rmk}

\begin{rmk}
It is natural to wonder whether or not the argument used in the proof of Theorem~\ref{UC general} yields to a Taylor expansion via nontrivial harmonic polynomials for both $u$ and $v$, close to each point of the zero level set $\{u=v=0\}$, as in \cite{CFJDE79, CFJDE85}. Such expansion would naturally provide a control on the Hausdorff dimension of its regular and singular set. However, in general such a Taylor expansion is not available, as shown by the counterexample $(x_1,0)$ which solves system \eqref{counter1}. It remains an open problem to establish the Taylor expansion for strongly coupled systems.
\end{rmk}

\section{Further discussion and open problems}\label{section discussion}

Let us start the section with a classification result which gives a sufficient condition to obtain \textit{proportionality of the components} of the solutions for the  system \eqref{gen syst}, and we give some applications. In what follows $\Omega$ will always be a \emph{bounded} domain.

Let $g:\Omega \times \R^2 \to \R$, $k>0$ be a constant, and $v$ be a solution of
\begin{align} \label{eq u}
\begin{array}{rclcc}
\Delta v +g(x,kv,v) &=& 0 &\mbox{in} & \Omega. \\
\end{array}
\end{align}
Now, suppose that $f:\Omega \times \R^2 \to \R$ is such that $f(x,kv,v) = kg(x,kv,v)$ for all $x\in \Omega$ and $v \in \R$. Then, it obvious that the pair $(kv,v)$ solves the system
\begin{align} \label{gen syst u=v}
\left\{
\begin{array}{rclcc}
\Delta u +f(x,u,v) &=& 0 &\mbox{in} & \Omega \\
\Delta v +g(x,u,v) &=& 0 &\mbox{in} & \Omega.
\end{array}
\right.
\end{align}

Here we give a condition on $f$ and $g$ which guarantees that any solution $u,v$ of \eqref{gen syst u=v} with $u=kv$ on $\partial \Omega$ is of the form $u=kv$ in $\Omega$. In this setting, the validity of the UCP is clearly reduced to the same problem for single equations.
We mention that this type of result has already been observed in \cite[eq.\ (1.9) and (1.12)]{MBS} and \cite[Section 2]{QuittnerSouplet}, see also \cite{Fa}.
For the sake of completeness we provide a (new) simpler proof which permits us to extend it to the fully nonlinear setting in Example \ref{FNU}. We start with the following technical lemma.

\begin{lem} \label{lemma condition}
Let $f,g:\Omega \times \R^2 \to \R$ be such that
\begin{equation}\label{condfg}
(u-kv)(f(x,u,v) - kg(x,u,v)) \leq 0 \ \textrm{ for all $x\in \Omega$ and } \ u, v \in \R,
\end{equation}
with $f, g$ continuous in the variables $u$ and $v$. Then $f(x,kt,t) = kg(x,kt,t)$ for all $x\in \Omega$ and $t \in \R$.
\end{lem}
\begin{proof}
Given $x\in \Omega$, $t\in \R$ and $\varepsilon>0$, set $u=k(t+\varepsilon)$, $v = t -\varepsilon$ to deduce $f(x,kt,t) \leq kg(x,kt,t)$. Then use $u=k(t-\varepsilon)$, $v = t +\varepsilon$ to conclude that $f(x,kt,t) = kg(x,kt,t)$.
\end{proof}

\begin{prop}\label{lemma u=v}
Suppose that $f,g:\Omega \times \R^2 \to \R$ satisfy \eqref{condfg}. Then, every solution $u,v$ of \eqref{gen syst u=v}, with $u=kv$ on $\partial\Omega$, is of the form $u=kv$ in $\Omega$.
\end{prop}
\begin{proof}
Consider the function $w=(u-kv)^2$. Then $w$ satisfies
\begin{center}
$\Delta w =2|Du-kDv|^2 +2(u-kv)\{-f(x,u,v)+kg(x,u,v)\}\geq 0 $\;\; in $\Omega$, \quad $w=0$ on $\partial\Omega$,
\end{center}
which implies that $w\leq 0$ in $\Omega$ by the maximum principle.
\end{proof}

\begin{example}[Hamiltonian type systems]\label{ex u=v H}
Every solution $u,v$ of the problem $-\Delta u = f(x,v)$, $-\Delta v = f(x,u)$ in $\Omega$, where $f(x,s)$ is nondecreasing with $s$, with $u=v$ on $\partial\Omega$, is of the form $u=v$ in $\Omega$. In particular, this is the case if $f(x,s) = a(x)|s|^{p-1}s$ with $a(x)\geq 0$ in $\Omega$ and $p>0$.
\end{example}

\begin{example}[Gradient type systems]\label{ex u=v G}
Every solution $u,v$ of the problem $-\Delta u = f(x,u) + \lambda(x)v$, $-\Delta v = f(x,v) + \lambda(x)u$ in $\Omega$, where $f(x,s)$ is nonincreasing with $s$ and $\lambda(x) \geq0$ in $\Omega$, with $u=v$ on $\partial\Omega$, is of the form $u=v$ in $\Omega$. In this case one example is $f(x,s) = -|s|^{p-1}s$, where $p>0$, and $\lambda$ is a positive constant $($Allen-Cahn type systems$)$.
\end{example}

\begin{example}[Fully nonlinear operators]\label{FNU}
Consider a fully nonlinear operator $F(x,X)$ satisfying \eqref{SC} with $\gamma=0$.
We infer that the conclusion of Proposition \ref{lemma u=v} remains the same for such operators, that is, every solution $u,v$ of the problem
$-F(x, D^2 u) = f(x,u,v)$, $-F(x, D^2 v) = g(x,u,v)$ in $\Omega$, with $u=k v$ on $\partial\Omega$, is of the form $u=kv$ in $\Omega$ provided \eqref{condfg} holds.
Indeed, since $D^2 w= 2 D(u-kv) \otimes D(u-kv) + 2 D^2(u-kv) (u-kv)$, using \eqref{def Pucci} and \cite[Lemma 2.10(5)]{CafCab}, we formally obtain
\begin{align*}
\mathcal{M}^+_{\lambda,\Lambda} (D^2 w) &\geq 2 \mathcal{M}^-_{\lambda,\Lambda}((Du-kDv)\otimes (Du-kDv))+2 \mathcal{M}^+_{\lambda,\Lambda}((D^2 u-kD^2 v)) \\
&\geq 2 \lambda |Du-kDv|^2 +2(u-kv)\{F(x,D^2 u)-kF(x,D^2 v)\},
\end{align*}
since the spectrum of $z\otimes z = (z_i z_j)_{ij}$, with $z= (z_1, \ldots, z_N)$, is $\{0,\dots,0,|z|^2\}$.
Observe that it is enough for $u,v$ to be merely sub and supersolutions respectively, with $u= kv$ on $\partial\Omega$.

In order to make sense to the calculations above, namely in the strong sense, we assume that $F$ has locally $W^{2,m}$ regularity of solutions. This is true for instance if $F$ is convex or concave in the $X$ entry. In this case one also implies $f(x,u(x),v(x)),\; g(x,u(x),v(x))\in L^m(\Omega)$. This is not strictly necessary since we could perform such arguments in the $C$-viscosity sense as long as we have $F,f,g$ continuous in the variable $x$; details are left to the interested reader.
\end{example}

As an application for the preceding proportionality of components, we consider the special scenario $p=q=1$, in what concerns a partial answer to the limiting case $pq=1$; more precisely, let $(u,v)$ be a solution to
\begin{align*}
-\Delta u  = {d} v, \quad
-\Delta v ={d} u  \quad \mbox{in} \ \Omega, \qquad
u = v \quad \mbox{on} \ \partial \Omega,
\end{align*}
with $d>0$ and $\Omega$ bounded. In this case $u=v$ in $\Omega$ by Proposition \ref{lemma u=v}. Therefore, UCP follows from UCP for a single equation. In particular, this implies the validity of UCP for the biharmonic problem
\begin{align*}
\Delta^2 u = d^2 u  \quad \mbox{in}  \ \Omega, \qquad
-\Delta u = d u  \quad \mbox{on} \  \partial \Omega.
\end{align*}

We stress that there are cases in which a UCP might follow from a UCP for scalar equations without necessarily the system being degenerate into a single equation. For instance, the Emden- Fowler systems \eqref{gen syst} with $f(x,u,v) = |x|^{\beta}|u|^{r-1}u |v|^{q-1}v$ and $g(x, u,v) = |x|^{\alpha}|u|^{p-1}u|v|^{s-1}v$ for nonnegative $p,q,r,s,\alpha, \beta$, in the special cases of either $r\geq1$ and $p>0$, or $s\geq 1$ and $q>0$. In particular, one can cover these systems in cases in which assumptions \eqref{growth}-\eqref{superlinear} are not satisfied.

To finish, it is worth mentioning that we still do not know how to prove nonexistence of solutions for the Lane-Emden system in a starshaped domain, namely the following conjecture.

\begin{conjecture}
Let \eqref{supercritico} hold and $\Omega$ be a bounded domain, strictly starshaped with respect to $0 \in \Omega$. Then the problem \eqref{LE} with $u=v=0$ on $\partial\Omega$ does not admit any nontrivial solution.
\end{conjecture}

Notice that if we had $Du \cdot Dv >0$ on $\partial\Omega$ -- as in the radial case, see Corollary \ref{corol DuDv>0} -- then the Poho\v{z}aev identity for systems, see equation (3.4) in \cite{Mitidieri93},
\[
\int_{\partial \Omega} Du \cdot D v  ( x \cdot \nu) \rmd S =\int_{\partial \Omega} \partial_\nu u\, \partial_\nu v  ( x \cdot \nu)  \rmd S = \left\{ \frac{N+\alpha}{p+1}+\frac{N+\beta}{q+1} -(N-2)\right\} \int_{\Omega} |x|^{\alpha}|u|^{p+1} \rmd x\leq 0,
\]
would imply the conjecture above immediately. The conjecture could also be proved if we knew $Du \cdot Dv \geq 0$ on $\partial\Omega$ in addition to a stronger form of UCP with a vanishing derivative on an open set of the boundary for only one of the functions $u,v$ -- again as in the radial case, see Theorem~\ref{UCradial}.

\medskip

\textbf{{Declarations of interest:}} none.

\end{document}